\documentclass[12pt]{amsart}
\usepackage{amsmath,amssymb,latexsym, amsthm, amscd, mathrsfs, stmaryrd} 
\usepackage[linktocpage=true]{hyperref}
\hypersetup{colorlinks,linkcolor=blue,urlcolor=cyan,citecolor=blue}

\usepackage[all]{xy}
\usepackage{hyperref}
\usepackage{eucal}
\usepackage{color}
\newcommand{\nc}{\newcommand}
\nc{\browntext}[1]{\textcolor{brown}{#1}}
\nc{\greentext}[1]{\textcolor{green}{#1}}
\nc{\redtext}[1]{\textcolor{red}{#1}}
\nc{\bluetext}[1]{\textcolor{blue}{#1}}
\nc{\brown}[1]{\browntext{ #1}}
\nc{\green}[1]{\greentext{ #1}}
\nc{\red}[1]{\redtext{ #1}}
\nc{\blue}[1]{\bluetext{ #1}}
\allowdisplaybreaks
\newtheorem{thm}{Theorem}  [section]

\newtheorem{prop}[thm]{Proposition}

\newtheorem{example}[thm]{Example}

\theoremstyle{remark}
\newtheorem{rem}[thm]{Remark}

\numberwithin{equation}{section}

\newcommand{\diag}{\text{diag}}

\newcommand{\mbf}{\mathbf}

\newcommand{\ev}{\bar{0}}

\nc{\fZ}{{\mathfrak Z}}
\newcommand{\z}{\CMcal Z}
\newcommand{\rT}{r_i (T_{\wb} E_i)}

\newcommand{\odd}{\bar{1}}

\newcommand{\ov}{\overline}
\newcommand{\qbinom}[2]{\begin{bmatrix} #1\\#2 \end{bmatrix} }
\newcommand{\Q}{\mathbb Q}

\newcommand{\U}{\mbf U}

\newcommand{\Ui}{{\mbf U}^\imath}

\newcommand{\vs}{\varsigma}

\newcommand{\Z}{\mathbb Z}

\newcommand{\B}{\mbf B}

\newcommand{\tK}{\widetilde{K}}
\def \I{\mathbb{I}}

\newcommand{\tUi}{\widetilde{{\mathbf U}}^\imath}

\newcommand{\tU}{\widetilde{\mathbf U}}

\def \btau{{{\tau}}}
\newcommand{\tk}{\widetilde{k}}

\newcommand{\Ib}{\I_{\bullet}}
\newcommand{\Iw}{\I_{\circ}}

\newcommand{\wb}{w_\bullet}

\nc{\Bin}{B_i^{(n)}}
\nc{\Binn}{B_i^{(2)}}
\nc{\Binnn}{B_i^{(3)}}

\def \tf{\widetilde{y}}

\def \bvs{{\boldsymbol{\varsigma}}}

\allowdisplaybreaks

\begin{document}

	\title[Serre-Lusztig relations for  $\imath$quantum groups II]{Serre-Lusztig relations for $\imath$quantum groups II}
	
	\author[Xinhong Chen]{Xinhong Chen}
	\address{Department of Mathematics, Southwest Jiaotong University, Chengdu 610031, P.R.China}
	\email{chenxinhong@swjtu.edu.cn}

	\author[Gail Letzter]{Gail Letzter}
	\address{Mathematics Research Group, National Security Agency, Fort Meade,
Maryland 20755-6844, 
and Department of Mathematics, University of Maryland, College Park, Maryland 20742}
	\email{gletzter@verizon.net}
	
	\author[Ming Lu]{Ming Lu}
	\address{Department of Mathematics, Sichuan University, Chengdu 610064, P.R.China}
	\email{luming@scu.edu.cn}

	\author[Weiqiang Wang]{Weiqiang Wang}
	\address{Department of Mathematics, University of Virginia, Charlottesville, VA 22904}
	\email{ww9c@virginia.edu}

	
	
\begin{abstract}
		The $\imath$Serre relations and the corresponding Serre-Lusztig relations are formulated and established for arbitrary $\imath$quantum groups arising from quantum symmetric pairs of Kac-Moody type. 
\end{abstract}
	
	\maketitle

	
	\section{Introduction}

	\subsection{}
	
In this Letter, we generalize some main results concerning $\imath$Serre relations in \cite{CLW18} and the corresponding Serre-Lusztig (i.e., higher order Serre) relations in \cite{CLW21} among Chevalley generators $B_i$ and $B_j$, for $\tau i =i = \wb i$ and $i \neq j \in \Iw$, in $\imath$quantum groups $\Ui$ or $\tUi$ arising from quantum symmetric pairs of {\em arbitrary} Kac-Moody type, to the general cases for $\tau i=i$ (with condition $\wb i=i$ dropped).  The notations are to be explained below, and we refer to {\em loc. cit.} for a more complete introduction of backgrounds.

\subsection{}
	
	We are concerned about the Serre type relations among the generators $B_i$ in an $\imath$quantum group $\Ui$ (or a universal $\imath$quantum group $\tUi$) arising from quantum symmetric pairs (QSP) $(\U, \Ui)$. Recall that the definition of QSP is built on the Satake diagrams or admissible pairs $(\I=\Iw\cup \Ib, \tau)$ \cite{Le99, Le03, Ko14}. A universal $\imath$quantum group $\tUi$ \cite{LW19a} is a subalgebra of a Drinfeld double quantum group $\tU$, and $\Ui$ with parameters are recovered from $\tUi$ by central reductions. 
	
	The Serre type relations of $\Ui$ were obtained by G.~Letzter \cite{Le03} in finite type. The Serre relations between $B_i, B_j$ (where $B_j$ appears in degree 1) were explicitly known \cite{Ko14, BK15} in an arbitrary $\imath$quantum group $\Ui$ (or $\tUi$) of Kac-Moody type, unless $\tau i=i \in \Iw$; in case $\tau i =i \in \Iw$, explicit Serre relations were written down under a strong constraint on the Cartan integers $|a_{ij}| \le 4$, cf. \cite{BB10, BK15}. General $\imath$Serre relations for $\tau i =i = \wb i \in \Iw$ and $j \in \Iw$ without any constraint on Cartan integers $a_{ij}$ have been formulated by the authors \cite{CLW18};  see \eqref{eq:S2} below.  ``Explicit" yet unwieldy formulas for Serre relations in an arbitrary $\imath$quantum group $\Ui$ are also obtained in \cite{DeC19}; the coefficients involved therein can be rather difficult to compute in practice. More recently, a compact presentation for arbitrary $\imath$quantum groups has been obtained in \cite{KY21} in terms of continuous $q$-Hermite polynomials and a new family of deformed Chebyshev polynomials. 
	
	The Serre-Lusztig (or higher order Serre) relations for $\tUi$ hold in closed forms \cite{CLW21}
	\begin{align} \label{eq:f}
		\tf_{i,j;n,m,\ov{p},\ov{t},e}=0,
	\end{align}
	for $\tau i =i = \wb i \in \Iw$, $i \neq j \in \Iw$ and $m\ge 1-na_{ij}$; see \eqref{eq:m-aodd}--\eqref{eq:m-aeven} for the definition of $\tf$. These are generalizations of higher order Serre relations for quantum groups in \cite{Lus94}. In case $n=1$ and $m=1-a_{ij}$, the above relation reduces to the $\imath$Serre relation  \cite{CLW18}
	\begin{align}
		\sum_{r=0}^{1-a_{ij}} (-1)^r  B_{i,\overline{a_{ij}}+\overline{p_i}}^{(r)}B_j B_{i,\overline{p}_i}^{(1-a_{ij}-r)} &=0,
		\label{eq:S2}
	\end{align}
	which holds for $\tau i =i = \wb i \in \Iw$, $i \neq j \in \Iw$ in an arbitrary $\imath$quantum group; more generally, the Serre-Lusztig relations of minimal degree (i.e., \eqref{eq:f} for $m=1-na_{ij}$ and $n\ge 1$) take a similar simple form as in \eqref{eq:S2}. These relations are expressed in terms of $\imath$divided powers $B_{i,\ov{p}}^{(m)}$ (cf. \cite{BW18a, BeW18, CLW21}), depending on a parity $\ov{p} \in \{\ev, \odd\}$.
	
	In particular, these relations hold for $i=\tau i$ in an arbitrary quasi-split $\imath$quantum group (i.e., when $\Ib =\emptyset$).
	Conjectures and examples for Serre-Lusztig relations of minimal degrees in $\imath$quantum groups of split affine ADE type in very different forms  were proposed earlier by Baseilhac and Vu \cite{BaV14, BaV15}; their conjecture was proved for $q$-Onsager algebra in \cite{Ter18}.

\subsection{}
	
	In this Letter, the $\imath$Serre relation for an $\imath$quantum group of arbitrary  Kac-Moody type will be formulated as 
\begin{align}
		\sum_{r=0}^{1-a_{ij}} (-1)^r  \B_{i,\overline{a_{ij}}+\overline{p_i}}^{(r)}B_j \B_{i,\overline{p}_i}^{(1-a_{ij}-r)} &=0.
		\label{eq:S2b}
\end{align}
That is, it formally takes the same form as \eqref{eq:S2}, where we have replaced $B_{i,\ov{p}}^{(m)}$ in \eqref{eq:S2} by a more general definition of $\imath$-divided powers $\B_{i,\ov{p}}^{(m)}$ defined in \eqref{eq:iDPodd}--\eqref{eq:iDPev} (that is, $\vs_i$ in $B_{i,\ov{p}}^{(m)}$ is replaced by $\vs_i \, \rT$); in the case when $\wb i=i$, $\B_{i,\ov{p}}^{(m)}$ is reduced to the original $B_{i,\ov{p}}^{(m)}$ thanks to $r_i(E_i)=1$. Moreover, the Serre-Lusztig relations \eqref{eq:f} are generalized accordingly to arbitrary $\imath$quantum group $\tUi$ (see Theorem~ \ref{thm:f=0}) and they follow by a recursive relation similar to the one in \cite{CLW21} (see Theorem~\ref{thm:recursion}).
	
	For $i\in \Iw$ with $\tau i =i \neq \wb i$, a version of $\imath$divided powers $B_{i}^{(m)}$ (independent of $\ov{p} \in \Z_2$) was introduced in \cite{BW21} as a key ingredient toward $\imath$canonical basis. These $B_{i}^{(m)}$ (for some suitable parameter $\vs_i$) therein satisfies a crucial integral property, i.e., it lies in the $\Z[q, q^{-1}]$-form of the modified $\imath$quantum group. In contrast, the $\B_{i,\ov{p}}^{(m)}$ introduced in this Letter are not integral for $m\ge 2$ and for any parameter $\vs_i$. Paraphrasing, the $\imath$divided powers for $\tau i=i \in \Iw$, arising in 2 totally different settings of $\imath$canonical basis and $\imath$Serre relations, miraculously coincide if {\em and only if} $\wb i=i$. (Alas, we had a mental block on the ``only if" part, and this explains why the formulation of this Letter were not noticed earlier when we were writing \cite{CLW21}.)

\subsection{}
	We shall establish in this Letter the $\imath$Serre relation \eqref{eq:S2b} and its corresponding Serre-Lusztig relations. Actually, we achieve much more, by formulating and establishing further generalizations of these relations, which involve higher powers $B_j^n$, for $n\ge 2$; see Theorem~\ref{thm:min}:
\begin{align*}
	\sum_{r+s=1-na_{ij}} (-1)^r \B^{(r)}_{i,\ov{p}} B_j^n \B_{i,\ov{p}+\ov{na_{ij}}}^{(s)} =0.
\end{align*}
This is referred to as Serre-Lusztig relations of minimal degrees. 
For more general Serre-Lusztig relations, see Theorem~ \ref{thm:f=0}. 
	
	The proof of the Serre-Lusztig relations in Theorem~\ref{thm:min} and Theorem~ \ref{thm:f=0} uses 2 key ingredients. First, it relies on the results in \cite{CLW21} in an essential way, and a reader is recommended to keep a copy of it at hand (as it is impractical for us to repeat verbatim most arguments of that paper in the current general setting). The other key ingredient is a new universality property for $\imath$quantum groups formulated in Proposition~\ref{prop:universalitystrong}. This universality statement follows by a (seemingly weaker) version of universality property proved in Proposition \ref{prop:universality} and then comparing with known Serre-Lusztig relations for quasi-split $\imath$quantum groups in \cite{CLW21}. (Such a universality result for $n=1$ was almost explicit in \cite{Le03, Ko14} based on their projection techniques, and was made very explicit in \cite{DeC19}.)  
	
This Letter is organized as follows. In Section~\ref{sec:iDP}, we set up the preliminaries on Drinfeld doubles and quantum symmetric pairs. Then we formulate the new $\imath$divided powers in $\imath$quantum groups.
In Section~\ref{sec:univ}, we establish a universality property in $\imath$quantum groups concerning the structure constants appearing in Serre-Lusztig relations. 
In Section~\ref{sec:SL}, we formulate and establish Serre-Lusztig relations in $\imath$quantum groups in a great generality. 

	\vspace{2mm}
\noindent {\bf Acknowledgement.}
	XC is supported by the Fundamental Research Funds for the Central Universities grant 2682020ZT100 and 2682021ZTPY043. ML is partially supported by the National Natural Science Foundation of China (No. 12171333). 
	WW is partially supported by the NSF grant DMS-2001351. We thank some anonymous experts for helpful feedbacks and suggestions. 
	
\section{$\imath$Divided powers in $\imath$quantum groups}
 \label{sec:iDP}
\subsection{New $\imath$divided powers for $\tUi$}
	
	Given a Cartan datum  $(\I,\cdot)$, we have a \emph{root datum} of type $(\I,\cdot)$ \cite[1.1.1, 2.2.1]{Lus94}, which consists of
\begin{itemize}
	\item[(a)] 
	two finitely generated free abelian groups $Y,X$ and a perfect bilinear pairing $\langle\cdot,\cdot\rangle:Y\times X\rightarrow\Z$;
	\item[(b)] 
	an embedding $\I\subset X$ ($i\mapsto \alpha_i$) and an embedding $\I\subset Y$ ($i\mapsto h_i$) such that $\langle h_i,\alpha_j\rangle =2\frac{i\cdot j}{i\cdot i}$ for all $i,j\in \I$.
\end{itemize}
	
	The matrix $C= (a_{ij})_{i,j\in \I} =(\langle h_i, \alpha_j\rangle)_{i,j\in \I}$ is a \emph{generalized Cartan matrix}. For $D=\diag(\epsilon_i\mid \epsilon_i \in \Z^+,\; i\in \I)$, where $\epsilon_i=\frac{i\cdot i}{2}$, $DC$ is symmetric. Let $q_i=q^{\epsilon_i}$ for any $i\in\I$. The associated Drinfeld-Jimbo quantum group $\U=\U_\I$ is a $\Q(q)$-algebra generated by $E_i, F_i, K_i^{\pm 1}$, for $i\in \I$. Let $W$ denote the Weyl group generated by simple reflections $s_i$ for $i\in\I$.
	
	The Drinfeld double $\tU =\tU_\I$ is a $\Q(q)$-algebra generated by $E_i, F_i, \tK_i, \tK_i'$, for $i\in \I$, and $\tK_i \tK_i'$ is central in $\tU$; cf., for example, \cite[\S6]{LW19a}. Then $\U$ is obtained from $\tU$ by a central reduction: 
	\[
	\U =\tU / (\tK_i \tK_i' -1\mid i\in\I).
	\]
	Let $\tU^+$ (and respectively, $\U^+$) be the subalgebra of $\tU$ (and respectively, $\U$) generated by $E_i$ $(i\in \I)$. Clearly, $\tU^+ \cong \U^+$, and we shall identify them. For $i\in\I$, denote by $r_{i}: \U^+ \rightarrow \U^+$, ${}_ir:\U^+\rightarrow \U^+$ the unique $\Q(q)$-linear maps \cite{Lus94}  such that
\begin{align}  \label{eq:rr}
\begin{split}
		r_{i}(1) = 0, \quad r_{i}(E_{j}) = \delta_{ij},
		\quad r_{i}(xx') = xr_{i}(x') + q^{i \cdot \mu'}r_{i}(x)x';
		\\
		{}_{i}r	(1) = 0, \quad {}_{i}r(E_{j}) = \delta_{ij},
		\quad {}_{i}r(xx') = {}_{i}r(x)x' + q^{\mu \cdot i}(x){}_{i}r(x'),
\end{split}
\end{align}
	for $x \in \U^+_{\mu}$ and $x' \in \U^+_{\mu'}$.
	
For any $x\in\U^+$, by \cite[Proposition 3.1.6]{Lus94} one has
	\begin{align}
		\label{eq:xFcomm}
		xF_i-F_ix=\frac{r_i(x)\tK_i-\tK_i'(_ir(x))}{q_i-q_i^{-1}}.
	\end{align}
	
	Let $(\I =\Iw \cup \Ib, \tau)$ be an admissible pair; cf. \cite[Definition 2.3]{Ko14}. Let $W_{\Ib}=\langle s_i\mid i\in\Ib\rangle$ be the   parabolic subgroup of $W$ with the longest element $\wb$.
	
	 Note that $\tU_{\Ib}$ (and respectively, $\U_{\Ib}$) is naturally a subalgebra of $\tU$ (and respectively, $\U$).  The $\imath$quantum group $\tUi$ is a (coideal) subalgebra of $\tU$ (see \cite{LW19a, CLW21}), which is generated by $\tU_{\Ib}$, $\tk_i =\tK_i \tK_{\tau i}'$  $(i\in \Iw)$, and
	\[
	B_i =F_i + T_{\wb} (E_{\tau i}) \tK_i' \qquad (i\in \Iw).
	\]
	Here $T_{w}=T''_{w,+1}$ denotes a braid group operator as in \cite{Lus94} for any $w\in W$. Note that $T_{\wb} (E_{\tau i}) \in \U^+ =\tU^+$. 
	
	In this Letter, we are mainly concerned about $B_i$, for $i\in \Iw$ with $\tau i=i$; in this case,
	$B_i =F_i + T_{\wb} (E_i) \tK_i' \in \tUi.$ It is known (cf. \cite{Ko14}) that $r_i(T_{w_\bullet} E_i) , {}_ir(T_{w_\bullet} E_i)\in \U_{\Ib}^+= \tU_{\Ib}^+ \subset \tUi$ and 
\begin{align}   \label{eq:rTB}
	[\rT, B_j]=0, \quad \text{for } i,j \in \Iw.
\end{align}
It follows by weight reason that
\begin{align}
	\label{eq:rTFE}
	[\rT,F_j]=0=[\rT, T_{w_\bullet} (E_{\tau j})\tK'_j],\quad \text{for } i, j\in\Iw.
\end{align}

	For $m\in \Z$, let $[m]_i =[m]_{q_i}$ denotes the quantum integer associate to $q_i$. Let $i\in \Iw$ with $\btau i= i$ (but we drop the assumption that $\wb i =i$ which was imposed in \cite{CLW18, CLW21}). The {\em $\imath${}divided powers} of $B_i$ in $\tUi$ are defined to be \begin{align}
		&& \B_{i,\odd}^{(m)}=\frac{1}{[m]_{i}!}
		\left\{ \begin{array}{ccccc}
			B_i\prod_{j=1}^k \Big(B_i^2 - [2j-1]_{i}^2 q_i \tk_i \, \rT \Big) & \text{if }m=2k+1,\\
			\prod_{j=1}^k \Big(B_i^2 - [2j-1]_{i}^2 q_i \tk_i \, \rT \Big) &\text{if }m=2k; \end{array}\right.
		\label{eq:iDPodd} \\
		\notag\\
		&& \B_{i,\ev}^{(m)}= \frac{1}{[m]_{i}!}
		\left\{ \begin{array}{ccccc}
			B_i\prod_{j=1}^k \Big(B_i^2 - [2j]_{i}^2 q_i \tk_i \, \rT\Big) & \text{if }m=2k+1,\\
			\prod_{j=1}^{k} \Big(B_i^2 - [2j-2]_{i}^2 q_i \tk_i \, \rT\Big) &\text{if }m=2k. \end{array}\right.
		\label{eq:iDPev}
\end{align}
	
	Given $\ov{p}\in \Z_2 =\{\ev, \odd\}$, the $\imath${}divided powers are determined by the following recursive relations, for $m\ge 0$:
\begin{eqnarray}
	\label{lem:dividied power}
	B_{i} \B_{i,\ov{p}}^{(m)}
	= \left\{\begin{array}{llll}
		&[m+1]_{i}\B_{i,\ov{p}}^{(m+1)}                                 & \text{if}\ \ov{p}\neq \ov{m}, \\
		&[m+1]_{i}\B_{i,\ov{p}}^{(m+1)} + [m]_{i} q_i \tk_i \, \rT \B_{i,\ov{p}}^{(m-1)} & \text{if}\ \ov{p}= \ov{m}.
	\end{array}\right.
\end{eqnarray}
We set $\B_{i,\ev}^{(m)}=0=\B_{i,\odd}^{(m)}$ for any $m<0$.

\subsection{New $\imath$divided powers for $\Ui$}
	
	The $\Q(q)$-algebra $\Ui =\Ui_\bvs$, for $\bvs =(\vs_i)_{i\in \Iw}$ (subject to some constraints \cite{BK15, BW21}), can be defined as a subalgebra of $\U$ (similar to $\tUi$ as a subalgebra of $\tU$). In particular, for $i\in \Iw$ with $\tau i=i$, we have $B_i =F_i +\vs_i T_{\wb} (E_i) K_i^{-1} \in \Ui.$ Alternatively, $\Ui$ is related to $\tUi$ by a central reduction:
\[
\Ui_\bvs =\tUi / \big(\tk_i -\vs_i \,\,(i=\tau i), \tk_i \tk_{\tau i} -\vs_i \vs_{\tau i} \,\,(i\neq \tau i) \big).
\]
	
	Let us specialize to the case $\tau i =i \in \Iw$, which is most relevant to us. Our parameter $\vs_i$ corresponds to the notation in \cite{BK15} as $\vs_i = - c_is(i)$. The parameters $c_i, s(i)$ therein were not needed separately. Similarly, the notation $\z_i = - s(i) \rT$ in \cite{BK15} is never needed separately, and instead $c_i \z_i$ and $\rT$ are all one needs. We have
\begin{align}    \label{eq:Bbz-1}
		c_i \z_i =\vs_i \, \rT.
\end{align}

	By a slight abuse of notation, the $\imath$divided powers in $\Ui$, denoted again by $\B_{i,\ov{p}}^{(m)}$, for $\ov{p}\in \Z_2$, are defined in almost the same way as in \eqref{eq:iDPodd}--\eqref{eq:iDPev}, with $\tk_i$ replaced by $\vs_i$:
\begin{align}
	&& \B_{i,\odd}^{(m)}=\frac{1}{[m]_{i}!}
	\left\{ \begin{array}{ccccc}
	B_i\prod_{j=1}^k \Big(B_i^2 - [2j-1]_{i}^2 q_i \vs_i \, \rT \Big) & \text{if }m=2k+1,\\
	\prod_{j=1}^k \Big(B_i^2 - [2j-1]_{i}^2 q_i \vs_i \, \rT \Big) &\text{if }m=2k; \end{array}\right.
		\label{eq:iDPodd2} \\
		\notag\\
	&& \B_{i,\ev}^{(m)}= \frac{1}{[m]_{i}!}
		\left\{ \begin{array}{ccccc}
			B_i\prod_{j=1}^k \Big(B_i^2 - [2j]_{i}^2 q_i \vs_i \, \rT\Big) & \text{if }m=2k+1,\\
			\prod_{j=1}^{k} \Big(B_i^2 - [2j-2]_{i}^2 q_i \vs_i \, \rT\Big) &\text{if }m=2k.     
		\end{array}\right.
		\label{eq:iDPev2}
\end{align}
	
\begin{rem}
	In the case when ${\wb} i =i$, we have $\rT=1$, and the $\imath$divided powers $\B_{i,\ov{p}}^{(m)}$ were introduced first in \cite{BW18a, BeW18} for a distinguished parameter $\vs_i=q_i^{-1}$ and they are $\imath$canonical basis elements in the modified $\imath$quantum group. The $\B_{i,\ov{p}}^{(m)}$ when ${\wb} i =i$ for a general parameter $\vs_i$ used in \cite{CLW18, CLW21} (denoted by $B_{i,\ov{p}}^{(m)}$) are obtained from the distinguished case above by a renormalization automorphism of $\Ui$.
		
	In case when ${\wb} i \neq i$ and then $\rT\neq 1$, $\B_{i,\ov{p}}^{(m)}$ in general do not lie in the $\Z[q,q^{-1}]$-form of $\Ui$. Toward the construction of $\imath$canonical basis, different $\imath$divided powers, $B_i^{(m)}$, which lie in the $\Z[q,q^{-1}]$-form of (modified) $\Ui$, for $\tau i =i \neq \wb i$, were introduced in \cite{BW21}.
\end{rem}

\section{A universality property for $\imath$quantum groups}
 \label{sec:univ}
 
 In this section, we shall establish a universality property on the structure constants appearing in the Serre-Lusztig relations.

\subsection{A  weak form of universality for $\Ui$}
  \label{subsec:A}

Recall the linear maps $r_i,  {}_i r$, for $i\in \I$ from \eqref{eq:rr}. For any $j,k\in\Iw$, by \eqref{eq:xFcomm}, we have
\begin{align}
\label{eq:FTEcom}
	F_jT_{w_\bullet}(E_{\tau k}) K_k^{-1}
		=q^{-k\cdot j}T_{w_\bullet}(E_{\tau k}) K_k^{-1} F_j
		+ \delta_{j,\tau k}Z_{\tau k} 
		+ \delta_{j,\tau k}Z_{\tau k}' K_k^{-2},
\end{align}
where we denote by $\mu$ the weight of ${}_{ k}r(T_{w_\bullet}E_{ k})$, and set
\begin{align}
\label{def:Z}
	Z_k=	\frac{r_{ k}(T_{w_\bullet}E_{ k})}{q_{ k}^{-1}-q_{ k}},\qquad 
	Z_k' =\frac{{}_{ k}r(T_{w_\bullet}E_{ k})}{q^{ k\cdot \mu}(q_{ k}-q_{ k}^{-1})}.
\end{align}
For any $\tau i=i\in\Iw$, by combining with \eqref{eq:rTFE}, we have 
\begin{align}
	\label{eq:ZZ}
	[Z_i,Z_j]=0=[Z_i,Z_j' K_j^{-2}],\quad \text{for }j\in\Iw.
\end{align}

For $\tau i =i \in \Iw$ and $i\neq j \in \Iw$, we denote
\begin{equation}  \label{eq:Bij3}
		S_{i,j;n} (B_i, B_j) := \sum_{r+s=1 -na_{ij}} (-1)^r  \qbinom{1 -na_{ij}}{r}_i B_i^{r} B_j^{n} B_i^{s}.
	\end{equation}
	
As a main step toward proving the universality property in Proposition \ref{prop:universalitystrong}, we first establish the following variant.

\begin{prop}
	\label{prop:universality}
Let $n\geq1$. For $\tau i=i\in\Iw$ and $i\neq j\in\Iw$, we have the following identity in $\Ui$:
\begin{align}  \label{eq:S=Cn}
		S_{i,j;n}(B_i, B_j) = \widehat C_{i,j;n}(B_i, B_j),
\end{align}
where $\widehat C_{i,j;n}(B_i, B_j)$ is a (non-commutative) polynomial in $B_i, B_j$ of the form
\begin{align}  \label{eq:Cn}
		\widehat C_{i,j;n}(B_i, B_j) =\sum_{r+s \le -1 -na_{ij}} \sum_{2m\leq n}\widehat{\varrho}_{r,s,m,n}^{(i,j,a_{ij})}\; \big(\vs_i \, Z_i \big)^{\frac{1 - na_{ij}-r-s}2} (\vs_jZ_{\tau j})^mB_i^r B_j^{n-2m} B_i^{s},
\end{align}
for some universal Laurent polynomials $\widehat{\varrho}_{r,s,m,n}^{(i,j,a_{ij})} \in \Z[q,q^{-1}]$ (i.e., they depend only on $a_{ij}$, $\epsilon_i$, $\epsilon_j$, and $r,s,m,n$). Moreover, if $\tau j\neq j$ then $\widehat{\varrho}_{r,s,m,n}^{(i,j,a_{ij})}=0$ for all $m>0$.
\end{prop}

\begin{proof}
	In the argument below, we assume that $j=\tau j$. (The $j\neq \tau j$ case follows from a similar analysis and can be obtained by setting $Z_{\tau j}=0$ and $Z_{\tau j}'=0$ everywhere.) Let $\U_{i,j}^-$ be the subalgebra generated by $F_i$ and $F_j$, and
\begin{align*}
	\mathbb U_{i,j}^+:=\sum_{\gamma\geq \alpha_i\text{ or }\gamma\geq \alpha_j} \U_\gamma^+\U_{i,j}^-\tU^0,\text{ and }\quad
	\mathbb T_{i,j}^+=\sum_{e,f\geq0,e+f>0}\U_{\Ib}\U_{i,j}^-K_i^{-2e}K_j^{-2f}.
\end{align*} 
By using \eqref{eq:FTEcom} and the definition $B_k=F_k+\vs_kT_{w_\bullet}E_{\tau k}K_k^{-1}$ for any $k\in\Iw$, one can expand out the term $B_i^aB_j^rB_i^b$ so that it satisfies
\begin{align*}
	B_i^aB_j^rB_i^b\in F_i^aF_j^rF_i^b+ \sum_{m,u,s,t}g_{i,j,a,b,r}^{m,u,s,t}(\vs_iZ_i)^{u+s+t}(\vs_jZ_j)^mF_i^{a-2u-t}F_j^{r-2m}F_i^{b-2s-t}+\mathbb U_{i,j}^+ + \mathbb T_{i,j}^+,
\end{align*}
where $g_{i,j,a,b,r}^{m,u,s,t}$ is a scalar. This result is a special case of \cite[Lemma 4.1]{Le19}. In fact, one starts with monomials in the $F_i, F_j, T_{w_\bullet}(E_i)K_i^{-1}$, and $T_{w_\bullet}(E_j)K_j^{-1}$. Terms of the form $T_{w_\bullet}(E_k)K_k^{-1}$ are moved to the left and new monomials are created via \eqref{eq:FTEcom} that have $Z_k$ or $Z_k'K_k^{-2}$ entries. Those monomials that have a $T_{w_\bullet}(E_k)K_k^{-1}$ that survives all the way on the left hand side of the expression (for either $k=i$ or $k=j$) become part of $\mathbb U_{i,j}^+$. Whenever a $Z_k'K_k^{-2}$ appears, we move the $K_k^{-2}$ to the right; if, in addition, there are no $T_{w_\bullet}(E_k)K_k^{-1}$ terms remaining all the way on the left, then the resulting monomial is in $ \mathbb T_{i,j}^+$. It follows from \eqref{eq:FTEcom}, \eqref{eq:rTFE} and \eqref{eq:ZZ} that each $g_{i,j,a,b,r}^{m,u,s,t}\in \Z[q,q^{-1}]$ and depends only on $a_{ij},\epsilon_i,\epsilon_j,a,b,m,u,s,t$.  
	
In the special case where $u=t=s=0$, each term of the form $(\vs_jZ_j)^mF_i^aF_j^{r-2m}F_i^b$ comes from expanding out $B_j^r$ inside the term $F_i^aB_j^rF_i^b$ and moving powers of $\vs_jZ_j$ to the left. Hence $g_{i,j,a,b,r}^{m,0,0,0}$ is independent of $a$ and $b$. 
	
Return to $S_{i,j;n}(B_i,B_j)$. From the above analysis, we have  
\begin{align*}
	\sum_{r+s=1-na_{ij}}(-1)^r\qbinom{1-na_{ij}}{r}_i\Big(F_i^rF_j^nF_i^s+\sum_{m} g_{i,j,r,s,n}^{m,0,0,0} (\vs_jZ_j)^mF_i^rF_j^{n-2m}F_i^s \Big)=0	
\end{align*}
by using the Serre-Lusztig relation and its non-standard variant (cf.  \cite[Corollary~ 3.3]{CLW21}). It follows that $S_{i,j;n}(B_i,B_j)$ is contained in the set 
\begin{align}
	\label{eq:reduce1}
	\sum_{u+v\leq -1-na_{ij}} \Big(\sum_m d_{u,v,m} (\vs_iZ_i)^{\frac{1-na_{ij}-u-v}{2}} (\vs_jZ_j)^mF_i^uF_j^{n-2m}F_i^v\Big)+\mathbb U_{i,j}^+ + \mathbb T_{i,j}^+,
\end{align}
where the coefficients $d_{u,v,m}$ come from sums of terms of the form $\qbinom{1-na_{ij}}{r}_ig_{i,j,r,s,n}^{\cdot,\cdot,\cdot,\cdot}$ and clearly are Laurent polynomials of the desired form. 
	
Let $z$ be the maximum of $u+v+n-2m$ with $d_{u,v,m}\neq0$. Replacing terms of the form $F_i^uF_j^{n-2m}F_i^v$ for $u+v+n-2m=z$ with $B_i^uB_j^{n-2m}B_i^v$ in \eqref{eq:reduce1} yields
\begin{align*}
	S_{i,j;n}(B_i,B_j)-&\sum_{u+v+n-2m=z} d_{u,v,m} (\vs_iZ_i)^{\frac{1-na_{ij}-u-v}{2}} (\vs_jZ_j)^mB_i^uB_j^{n-2m}B_i^v
		\\
	\in& \sum_{u+v+n-2m<z} d'_{u,v,m} (\vs_iZ_i)^{\frac{1-na_{ij}-u-v}{2}} (\vs_jZ_j)^mF_i^uF_j^{n-2m}F_i^v+\mathbb U_{i,j}^+ + \mathbb T_{i,j}^+,
\end{align*}
where once again the coefficients $d'_{u,v,m}$ have the desired form. Repeating this process and noting that $\Ui_\bvs\cap \big(\mathbb U_{i,j}^+ + \mathbb T_{i,j}^+\big)=0$, we have proved the proposition.
\end{proof}

\subsection{The universality for $\Ui$}

Using the Serre-Lusztig relations for quasi-split $\imath$quantum groups obtained in \cite{CLW21}, we can sharpen the statement in Proposition~ \ref{prop:universality}.

\begin{prop} [Universality]
	\label{prop:universalitystrong}
Let $n\geq1$. For $\tau i=i\in\Iw$ and $i\neq j\in\Iw$, we have the following identity in $\Ui$:
\begin{align}  \label{eq:S=Cnnew}
		S_{i,j;n}(B_i, B_j) = C_{i,j;n}(B_i, B_j),
\end{align}
where $C_{i,j;n}(B_i, B_j)$ is a (non-commutative) polynomial in $B_i, B_j$ of the form
\begin{align}  \label{eq:Cnnew}
	C_{i,j;n}(B_i, B_j) =\sum_{r+s \le -1 -na_{ij}} \varrho_{r,s|n}^{(i,j,a_{ij})}\; \big(\vs_i \, r_i(T_{w_\bullet}E_i) \big)^{\frac{1 - na_{ij}-r-s}2} B_i^r B_j^{n} B_i^{s},
\end{align}
for some universal Laurent polynomials $\varrho_{r,s|n}^{(i,j,a_{ij})}\in \Z[q,q^{-1}]$ (which depend only on $a_{ij}$, $\epsilon_i$, $\epsilon_j$ and $r,s,n$). 
\end{prop}
\noindent (The identities \eqref{eq:S=Cnnew}--\eqref{eq:Cnnew} remain valid in $\tUi$ when $\vs_i$ in $C_{i,j;n}(B_i, B_j)$ above is replaced by $\tk_i$.)

\begin{proof}
	By Proposition~\ref{prop:universality}, we have a Serre-Lusztig type relation (for $i\neq j \in \Iw$) of the form \eqref{eq:S=Cn}--\eqref{eq:Cn},  $S_{i,j;n}(B_i, B_j) =\widehat C_{i,j;n}(B_i, B_j)$, where $\widehat C_{i,j;n}(B_i, B_j)$ is an expression with universal coefficients $\widehat{\varrho}_{r,s,m,n}^{(i,j,a_{ij})} \in \Z[q,q^{-1}]$, which depend only on $a_{ij}$, $\epsilon_i$, $\epsilon_j$, and $r,s,m,n$. 
	
\vspace{2mm}
{\bf Claim ($\star$).} We have $\widehat{\varrho}_{r,s,m,n}^{(i,j,a_{ij})} =0$, for $m>0$. 
\vspace{2mm}
	
Let us prove the Claim. In order to determine $\widehat{\varrho}_{r,s,m,n}^{(i,j,a_{ij})}$, we shall restrict ourselves to the setting of quasi-split $\imath$quantum groups, where $r_i(T_{w_\bullet}E_i)=1$, i.e., $Z_i=\frac{1}{q_i^{-1}-q_i}$. By \cite[Theorem 4.1]{CLW21}, for quasi-split $\imath$quantum groups $\tUi$, we have 
\begin{align}
	\label{eq:SerreBn1}
\sum_{a+b=1-na_{ij}} (-1)^r B^{(a)}_{i,\ov{p}} B_j^n B_{i,\ov{p}+\ov{na_{ij}}}^{(b)}=0, 
\quad (n\geq 0),
\end{align}
where the $\imath$divided powers $B_{i,\ov{p}}^{(m)}$ are as in \eqref{eq:iDPodd2}-\eqref{eq:iDPev2} but with $r_i(T_{w_\bullet}E_i)=1$.

Expanding $[1-na_{ij}]_i!\cdot$LHS\eqref{eq:SerreBn1} into a linear combination of monomials of the form $B_i^r B_j^n B_i^s$, we have
\begin{align}  \label{eq:Cnnewqs}
S_{i,j;n}(B_i, B_j) =\sum_{r+s \le -1 -na_{ij}} \varrho_{r,s|n}^{(i,j,a_{ij})}\; 
\vs_i^{\frac{1 - na_{ij}-r-s}2} B_i^r B_j^{n} B_i^{s},
\end{align}
for some universal Laurent polynomials $\varrho_{r,s|n}^{(i,j,a_{ij})}\in \Z[q,q^{-1}]$ (i.e., they depend only on $a_{ij}$, $\epsilon_i$, $\epsilon_j$ and $r,s,n$). As \eqref{eq:Cnnewqs} does not 
involve $\vs_j Z_{\tau j}'$, , a comparison of \eqref{eq:Cnnewqs} and \eqref{eq:S=Cn}--\eqref{eq:Cn} with $Z_i=\frac{1}{q_i^{-1}-q_i}$  shows that 
$\widehat{\varrho}_{r,s,m,n}^{(i,j,a_{ij})} =0$, for $m>0$. Claim ($\star$) is proved. 

A comparison of \eqref{eq:Cnnewqs} and \eqref{eq:S=Cn}--\eqref{eq:Cn} with $Z_i=\frac{1}{q_i^{-1}-q_i}$ again further shows that 
\begin{align}
 \label{eq:rho2}
\widehat{\varrho}_{r,s,0,n}^{(i,j,a_{ij})} =\varrho_{r,s|n}^{(i,j,a_{ij})}\cdot (q_i^{-1}-q_i)^{\frac{ na_{ij}+r+s-1}2},
\quad \text{for all } r,s,n.
\end{align}

Now back to arbitrary $\imath$quantum groups $\tUi$. 
By Claim ($\star$) and \eqref{eq:rho2}, we can rewrite the identity \eqref{eq:S=Cn}--\eqref{eq:Cn} in the  precise form of the Serre-Lusztig relation \eqref{eq:S=Cnnew}--\eqref{eq:Cnnew}. The proposition is proved. 
\end{proof}
	
\begin{rem}
		For $n=1$, the universality result in Proposition \ref{prop:universalitystrong} has been (somewhat implicitly) known in \cite{Le03, Ko14} and made explicit in \cite{DeC19}. Actually, a very careful and tedious work was carried out in \cite{DeC19} to describe explicitly these universal polynomials $\varrho_{r,s|1}^{(i,j,a_{ij})}$ in $C_{i,j;1}(B_i,B_j)$; see \cite[Theorem~4.7]{DeC19}. 
\end{rem}

	
\section{The Serre-Lusztig relations in $\imath$quantum groups} 
	 \label{sec:SL}

\subsection{Serre-Lusztig relations of minimal degree} 
	
	We consider $\imath$quantum groups $\Ui$ of arbitrary Kac-Moody type, where $\Ib \neq \emptyset$ is allowed.
	
\begin{thm}
[Serre-Lusztig relations  of minimal degree]
	\label{thm:min}
For any $i \neq j\in \Iw$ such that $\tau i = i$ and $\ov{t} \in\Z_2$, the following identities hold in $\tUi$ for $n\geq 1$:
\begin{align}
\sum_{r+s=1-na_{ij}} (-1)^r \B^{(r)}_{i,\ov{p}} B_j^n \B_{i,\ov{p}+\ov{na_{ij}}}^{(s)} &=0, 
		\label{eq:SerreBn12}
			\\
\sum_{r+s=1-na_{ij}} (-1)^r \B^{(r)}_{i,\ov{p}} B_{j, \ov{t}}^{(n)} \B_{i,\ov{p}+\ov{na_{ij}}}^{(s)} &=0.
		\label{eq:SerreBn10}
\end{align}
\end{thm}

	We do not recall the precise formulas for the $\imath$divided powers $B_{j, \ov{t}}^{(n)}$ in 3 cases, and we refer to \cite[(5.12)]{BW21} and \cite[(5.5)]{CLW21} for details.
	
		For $n=1$, the identity \eqref{eq:SerreBn12} reduces to the $\imath$Serre relation \eqref{eq:S2b} in $\tUi$. In case ${\wb} i =i$ and thus $\rT=1$, the relations in Theorem~\ref{thm:min} reduce to \cite[Theorem~A]{CLW21}, and the $\imath$Serre relation \eqref{eq:S2b} was obtained in \cite{CLW18}.

\begin{proof}[Proof of Theorem~\ref{thm:min}]
	As explained in \cite[Introduction]{CLW21},  \eqref{eq:SerreBn10} follows from  \eqref{eq:SerreBn12} by \cite[Proposition~3.2]{CLW21}. Hence it suffices to prove \eqref{eq:SerreBn12}, or its $\Ui$-variant, where $\tk_i$ is replaced by $\vs_i$ in the $\imath$divided powers.
		
By Proposition~\ref{prop:universalitystrong}, we have a Serre-Lusztig relation (for $i\neq j \in \Iw$) of the form \eqref{eq:S=Cnnew}, i.e., $S_{i,j;n}(B_i, B_j) =C_{i,j;n}(B_i, B_j)$. 
		
The above discussion remains valid in the setting of quasi-split $\imath$quantum groups where $\rT=1$; in this case, we already have a Serre-Lusztig relation (for $i\neq j \in \Iw$) of the form \eqref{eq:SerreBn12}, where $\vs_i\rT$ reduces to $\vs_i$ in the definition of $\B_{i,\ov{p}}^{(r)}$; see \cite[(3.9)]{CLW18}. Then we have the following expansion in terms of (non-commutative) monomials in $B_i, B_j$, where it is understood that $\rT=1$ in the $\imath$divided powers and in $C_{i,j;n}(B_i, B_j)$:
\begin{align}
\label{eq:BBSC}
	\sum_{r=0}^{1-na_{ij}} (-1)^r & \B_{i,\overline{na_{ij}}+\overline{p_i}}^{(r)}B_j^n \B_{i,\overline{p}_i}^{(1-a_{ij}-r)} |_{\rT=1} \\
	&=[1-na_{ij}]_i!^{-1} \Big( S_{i,j;n}(B_i,B_j) -C_{i,j;n}(B_i, B_j)|_{\rT=1} \Big).
	\notag
\end{align}
Indeed, the universal polynomials $\varrho_{r,s|n}^{(i,j,a_{ij})}$ appearing in $C_{i,j;n}(B_i, B_j)$ from \eqref{eq:Cnnew} are determined from the expansion of the LHS above as in Proposition \ref{prop:universalitystrong}.
		
Return to the setting of general $\imath$quantum groups. The above formula \eqref{eq:BBSC} remains valid when replacing the scalar $\vs_i$ by $\vs_i \rT =c_i \z_i$ (which can be regarded as a commuting variable by \eqref{eq:rTB} when dealing with these relations) on both sides. The effect of such replacement is the removal of the restriction $|_{\rT=1}$ on both sides of \eqref{eq:BBSC}, that is, the following identity holds:
\begin{align}
 \label{eq:BBSC2}
	\sum_{r=0}^{1-na_{ij}} (-1)^r  \B_{i,\overline{na_{ij}}+\overline{p_i}}^{(r)}B_j^n \B_{i,\overline{p}_i}^{(1-a_{ij}-r)}  
	&=[1-na_{ij}]_i!^{-1} \Big( S_{i,j;n}(B_i,B_j) -C_{i,j;n}(B_i, B_j)\Big).
\end{align}
Since RHS\eqref{eq:BBSC2} $=0$ by \eqref{eq:S=Cnnew},  the identity \eqref{eq:SerreBn12} follows.
\end{proof}
	
\begin{rem}
Note that it is possible to deduce the formulas \eqref{eq:SerreBn12}--\eqref{eq:SerreBn10} in the $n=1$ case directly from  \cite[Proposition~4.6, Theorem~4.7]{DeC19}, though the proofs we provide here do not rely on the formulas of \cite{DeC19}.  Moreover,  it is instructive to compare the $\imath$Serre relation in a canonical form \eqref{eq:S2b} (as in \cite{CLW18}) to a  complicated formulation in \cite[Theorem~4.7]{DeC19}. They coincide up to a scalar multiple of $[1-a_{ij}]_i!$.  
\end{rem}

\begin{rem}
	\label{rem:closed}
When the parameter $\vs_i$ satisfies the conditions in \cite[(3.7)]{BW21} (which goes back to \cite{BK15} in some form), it follows from \cite[(5.10)]{BW21} (or \cite[Theorem~3.11(2)]{BK15}) that $q_i \vs_i \rT$ is bar invariant. Hence the relation \eqref{eq:S2b} in the $\Ui$ setting is manifestly bar invariant in this case. Such a bar invariance was also observed  in \cite{DeC19} based on the explicit formulas therein.
\end{rem}

\begin{rem}
	Combining the relations obtained by Letzter, Kolb and Balagovic (cf., e.g., \cite{Le03, Ko14, BK15}) and \eqref{eq:S2b} in this Letter, all the Serre type relations between $B_i$ and $B_j$ for $\imath$quantum groups $\tUi$ (or $\Ui$) of arbitrary Kac-Moody type have been formulated in {\em clean and closed} formulas in terms of $\imath$divided powers, except in the case when $\tau(i) =i \in \Iw$ and $j\in \Ib$; examples in this exceptional case can be found in \cite{BK15}. For a different expression of defining relations in this case, see \cite{KY21}.
\end{rem}
	
\begin{example}
	\label{ex:white}
Let  $j\neq i\in \Iw$ with $\tau i=i$. With the help of \eqref{eq:Bbz-1}, the formula \eqref{eq:S2b} specializes to
	\begin{enumerate}
	\item
		$B_i^2B_j -[2]_i B_iB_jB_i +B_jB_i^2 =q_i c_i\z_i B_j$, for $a_{ij}=-1$,
	\item
		$B_i^3B_j -[3]_i B_i^2B_jB_i +[3]_i B_iB_jB_i^2 -B_jB_i^3 = [2]_i^2 q_i c_i\z_i (B_iB_j - B_jB_i)$, for $a_{ij}=-2$.
\end{enumerate}
The above two formulas, together with a formula for $a_{ij}=-3$, were earlier obtained  in \cite[Theorem~ 3.7(2)]{BK15} by rather involved computations.
\end{example}

\subsection{Definition of $\tf_{i,j;n,m,\ov{p},\ov{t},e}$ and $\tf_{i,j;n,m,\ov{p},\ov{t},e}'$}
	
Let $i\neq j\in \Iw$ be such that $\btau i=i$. For $m\in\Z$, $n\in\Z_{\geq0}$, $e=\pm1$ and $\ov{p},\ov{t} \in \Z_2$, we define elements $\tf_{i,j;n,m,\ov{p},\ov{t},e}$ and $\tf_{i,j;n,m,\ov{p},\ov{t},e}'$ in $\tUi$ below, depending on the parity of $m-na_{ij}$. (They are simply modified from those in the same notations in \cite{CLW21}, with a substitution of $q_i\tk_i$ by $q_i\tk_i \rT$.)
	
If $m-na_{ij}$ is odd, we let
\begin{align}
		\label{eq:m-aodd}
& \tf_{i,j;n,m,\ov{p},\ov{t},e} =
		\\
& \sum_{u\geq0}(q_i\tk_i \rT)^{u} \Big\{
	\sum_{\stackrel{ r+s+2u=m}{\ov{r}=\ov{p}+\ov{1}}}
	(-1)^r q_i^{-e((m+na_{ij})(r+u)-r)}\qbinom{\frac{m+na_{ij}-1}{2}}{u}_{q_i^2} \B^{(r)}_{i,\ov{p}} B_{j,\ov{t}}^{(n)} \B^{(s)}_{i,\ov{p}+\ov{na_{ij}}}
		\notag \\ \notag
&\quad
	+\sum_{\stackrel{ r+s+2u=m}{\ov{r}=\ov{p}}}
	(-1)^{r}q_i^{-e((m+na_{ij}-2)(r+u)+r)}\qbinom{\frac{m+na_{ij}-1}{2}}{u}_{q_i^2} \B^{(r)}_{i,\ov{p}} B_{j,\ov{t}}^{(n)} \B^{(s)}_{i,\ov{p}+\ov{na_{ij}}}\Big\} ;
\end{align}
if $m-na_{ij}$ is even, then we let
\begin{align}
\label{eq:m-aeven}
& \tf_{i,j;n,m,\ov{p},\ov{t},e} =
		\\
&\sum_{u\geq0}(q_i\tk_i \rT)^{u}
	\Big\{
	\sum_{\stackrel{ r+s+2u=m}{ \ov{r}=\ov{p} +\ov{1}}}
	(-1)^{r} q_i^{-e(m+na_{ij}-1)(r+u)}
	\qbinom{\frac{m+na_{ij}}{2}}{u}_{q_i^2} \B^{(r)}_{i,\ov{p}} B_{j,\ov{t}}^{(n)} \B^{(s)}_{i,\ov{p}+\ov{na_{ij}}}
	\notag \\
&\quad +\sum_{\stackrel{ r+s+2u=m} {\ov{r}=\ov{p} }}
	(-1)^{r} q_i^{-e(m+na_{ij}-1)(r+u)}
	\qbinom{\frac{m+na_{ij}-2}{2}}{u}_{q_i^2} \B^{(r)}_{i,\ov{p}} B_{j,\ov{t}}^{(n)} \B^{(s)}_{i,\ov{p}+\ov{na_{ij}}}\Big\}. \notag
\end{align}

If $m-na_{ij}$ is odd, we let
\begin{align}
	\label{eq:evev1}
	& \tf_{i,j;n,m,\ov{p},\ov{t},e}'  =
		\\
	&\sum_{u\geq0}(q_i\tk_i \rT)^{u} \Big\{
	\sum_{\stackrel{r+s+2u=m}{\ov{r}=\ov{p}+\ov{1}}}
	(-1)^{r}q_i^{-e((m+na_{ij})(r+u)-r)}\qbinom{\frac{m+na_{ij}-1}{2}}{u}_{q_i^2} \B^{(s)}_{i,\ov{p}} B_{j,\ov{t}}^{(n)} \B^{(r)}_{i,\ov{p}+\ov{na_{ij}}}
		\notag \\ \notag
	& +\sum_{\stackrel{ r+s+2u=m}{\ov{r}=\ov{p}}}
	(-1)^{r}q_i^{-e((m+na_{ij}-2)(r+u)+r)}\qbinom{\frac{m+na_{ij}-1}{2}}{u}_{q_i^2} \B^{(s)}_{i,\ov{p}} B_{j,\ov{t}}^{(n)} \B^{(r)}_{i,\ov{p}+\ov{na_{ij}}}\Big\};
\end{align}
if $m-na_{ij}$ is even, then we let
\begin{align}
	\label{eq:evev0}
	& \tf_{i,j;n,m,\ov{p},\ov{t},e}' =
		\\
	& \sum_{u\geq0}(q_i\tk_i \rT)^{u}\Big\{
	\sum_{\stackrel{ r+s+2u=m}{\ov{r}=\ov{p} +\ov{1}}}
	(-1)^{r} q_i^{-e(m+na_{ij}-1)(r+u)}
	\qbinom{\frac{m+na_{ij}}{2}}{u}_{q_i^2} \B^{(s)}_{i,\ov{p}} B_{j,\ov{t}}^{(n)} \B^{(r)}_{i,\ov{p}+\ov{na_{ij}}}
		\notag \\
	&+\sum_{\stackrel{ r+s+2u=m}{ \ov{r}=\ov{p}}}
	(-1)^{r} q_i^{-e(m+na_{ij}-1)(r+u)}
	\qbinom{\frac{m+na_{ij}-2}{2}}{u}_{q_i^2} \B^{(s)}_{i,\ov{p}} B_{j,\ov{t}}^{(n)} \B^{(r)}_{i,\ov{p}+\ov{na_{ij}}}\Big\}. \notag
\end{align}

\subsection{Serre-Lusztig relations in $\tUi$} 
Denote
\begin{equation*}
	f_{i,j;n,m,e}^- = \sum_{r+s=m} (-1)^r q_i^{er(1-na_{ij}-m)} F_i^{(r)} F_j^{(n)} F_i^{(s)},
\end{equation*}
for $i\neq j \in \I$, $n>0$, and $e=\pm 1$.
The following Serre-Lusztig relations hold in the quantum group $\U$ (cf. \cite{Lus94}):
$
f_{i,j;n,m,e}^- =0,\text{ for } m \geq 1 -na_{ij}.
$
		
The Serre-Lusztig relations as formulated in \cite[Theorems~B, C, D]{CLW21} for quasi-split $\imath$quantum groups  (upon a substitution of $\tk_i$ by $\tk_i \rT$ in $\tUi$ or $\vs_i \rT$ in $\Ui$) remain valid for arbitrary $\imath$quantum groups; see Theorems~\ref{thm:nstd}, \ref{thm:recursion} and \ref{thm:f=0} below.  The same proofs {\em loc. cit.} (upon a substitution of a ``scalar" $\tk_i$ by another ``scalar" $\tk_i \rT$ in $\tUi$ as far as these relations are concerned), which are based on Serre-Lusztig relations of minimal degree (see Theorem~\ref{thm:min}), go through verbatim in the current setting. We shall not repeat the long proofs here. 
	
	As the identities \eqref{eq:SerreBn12}--\eqref{eq:SerreBn10} (i.e., Theorem~\ref{thm:min}) hold for $n\geq1$, so are the statements in Theorem~\ref{thm:nstd} and Theorem~\ref{thm:f=0} below. 
	
\begin{thm}  \label{thm:nstd}
For any $i \neq j\in \Iw$ such that $\btau i=i$,$u \in \Z_{\ge 0}$, and $\ov{t} \in\Z_2$, the following identities hold in $\tUi$ (or in $\Ui$),  for $n\geq1$:
\begin{align}
	\sum_{r+s=1-na_{ij}+2u} (-1)^r \B^{(r)}_{i,\ov{p}} B_j^n \B_{i,\ov{p}+\ov{na_{ij}}}^{(s)} &=0,
			\label{eq:Serre2t}
			\\
	\sum_{r+s=1-na_{ij}+2u} (-1)^r \B^{(r)}_{i,\ov{p}} B_{j,\ov{t}}^{(n)} \B_{i,\ov{p}+\ov{na_{ij}}}^{(s)} &=0.
			\label{eq:Serre1t}
\end{align}
\end{thm}

\begin{thm}  \label{thm:recursion}
		For $i\neq j\in \Iw$ such that  $\btau i=i$, $\ov{p}, \ov{t} \in\Z_2$, $n\ge 0$, and $e=\pm1$, the following identity holds in $\tUi$:
\begin{align}
	&q_i^{-e(2m+na_{ij})}  B_i\tf_{i,j;n,m,\ov{p},\ov{t},e}-\tf_{i,j;n,m,\ov{p},\ov{t},e}B_i\\ \notag
	&\quad = -[m+1]_{i} \tf_{i,j;n,m+1,\ov{p},\ov{t},e}
	+[m+na_{ij}-1]_{i} q_i^{1-e(2m+na_{ij}-1)} \tk_i\rT \tf_{i,j;n,m-1,\ov{p},\ov{t},e}.
\end{align}
\end{thm}

\begin{thm} [Serre-Lusztig relations]
		\label{thm:f=0}
Let $i\neq j\in \Iw$ such that  $\btau i=i$, $\ov{p}, \ov{t} \in\Z_2$, $n\ge 0$, and $e=\pm1$. Then, for $m<0$ and $m>-na_{ij}$, the following identities hold in $\tUi$,  for $n\geq1$:
\begin{align}
		\tf_{i,j;n,m,\ov{p},\ov{t},e}=0,\qquad \tf_{i,j;n,m,\ov{p},\ov{t},e}'=0.
\end{align}
\end{thm}
	
\begin{rem}
Theorems~\ref{thm:recursion} and \ref{thm:f=0} hold  if we replace $B_{j,\ov{t}}^{(n)}$ by $B_j^n$ throughout the definitions of $\tf_{i,j;n,m,\ov{p},\ov{t},e}$ and $\tf_{i,j;n,m,\ov{p},\ov{t},e}'$ in \eqref{eq:m-aodd}--\eqref{eq:evev0}. Theorems~\ref{thm:recursion} and \ref{thm:f=0} remain valid over $\Ui =\Ui_\bvs$, once we replace $\tk_i$ by $\vs_i$ in the definition of $\tf_{i,j;n,m,\ov{p},\ov{t},e}$.
\end{rem}

	

\end{document}